\documentclass[oneside,english]{amsart}
\usepackage[T1]{fontenc}
\usepackage[latin9]{inputenc}
\usepackage{amsthm}
\usepackage{amssymb}
\usepackage{graphicx}
\usepackage{esint}

\makeatletter

\newcommand{\noun}[1]{\textsc{#1}}

\numberwithin{equation}{section}
\numberwithin{figure}{section}
\theoremstyle{plain}
\newtheorem{thm}{\protect\theoremname}
  \theoremstyle{definition}
  
  \theoremstyle{plain}
  \newtheorem{lem}[thm]{\protect\lemmaname}
  \theoremstyle{plain}
  
  \theoremstyle{remark}

\makeatother

\usepackage{babel}
  \providecommand{\definitionname}{Definition}
  \providecommand{\lemmaname}{Lemma}
  \providecommand{\propositionname}{Proposition}
  \providecommand{\remarkname}{Remark}
\providecommand{\theoremname}{Theorem}

\def\p{\partial}

\def\e{\varepsilon}
\def\Ni{N_{\rm in}}
\def\No{N_{\rm out}}
\def\nui{\nu_{\rm in}}
\def\nuo{\nu_{\rm out}}
\def\Pin{\mathsf{P}_{\mathrm{in}}}
\def\Pout{\mathsf{P}_{\mathrm{out}}}

\def\bea{\begin{eqnarray}}
\def\eea{\end{eqnarray}}
\def\beastar{\begin{eqnarray*}}
\def\eeastar{\end{eqnarray*}}

\begin{document}

\title{Hardy Spaces and Boundary Conditions from the Ising Model}

\author{Cl\'{e}ment Hongler and Duong H. Phong}
\address{Department of Mathematics, Columbia University. 2990 Broadway, New
York, NY 10027, USA.}

\email{hongler@math.columbia.edu}
\email{phong@math.columbia.edu}

\begin{abstract}

Functions in Hardy spaces on multiply-connected domains
in the plane are given an explicit characterization in terms of a boundary condition inspired by the two-dimensional
Ising model. The key underlying property is the positivity of a certain
operator constructed inductively on the number of components of the boundary.

\end{abstract}

\maketitle

\section{Introduction}

A remarkable property of critical phenomena in two dimensions is their local conformal invariance. This has resulted in a rich interaction between statistical
physics and many branches of mathematics, including probability, complex analysis, Riemann surfaces, and infinite-dimensional Lie algebras (see e.g.
\cite{BPZ, DJKM, DiF, DP, Sch07, Smi06, Smi10b} and references therein).

\smallskip

The goal of the present paper is to show how ideas from two-dimensional statistical
physics can help answering an important question in complex analysis,
namely how to explictly characterize the boundary values of holomorphic functions on a smooth multiply-connected domain $\Omega$. See \cite{Fis} for an elementary introduction to analytic function spaces on planar domains.

\smallskip
The simply-connected case is well-known. In this case, $\Omega$ can be assumed to be the unit disk $\mathbb{D}$ by the Riemann mapping theorem. 
The space of holomorphic functions admitting $L^2$ boundary values is the  
Hardy space $H^2(\mathbb{D})$, and their boundary values can be characterized by the condition that all their Fourier coefficients of negative index vanish
(see e.g. \cite{Z}). The projection operator on the space of $L^2$ boundary values of holomorphic functions
is given in terms of the Hilbert transform, which is the primary example of a singular integral operator.
Underlyingl this is the basic fact that, for any real-valued $L^2$ function
$f$ on $\p\mathbb{D}$, there is a single-valued holomorphic function $F$ on $\mathbb{D}$ with $f$ as its real part on $\p\mathbb{D}$.

\smallskip
In general, there is no such function $F$ if $\Omega$ is multiply-connected (see \cite{EaMa}, for instance). However, a natural remedy to this 
is the following boundary value problem, which (as we will show) always possesses a unique solution:
\begin{eqnarray}
\label{boundarycondition}
F \ \ {\rm holomorphic} \ \ {\rm on}\ \ \Omega,
\qquad
\Im \mathfrak{m} ((F-f)\nu^{{1\over 2}})=0
\ \ {\rm on}\ \ \p\Omega
\end{eqnarray}
where $\nu$ is the outer normal to the boundary, viewed as a complex number. This boundary value problem is suggested by recent studies of the  Ising model \cite{Smi10a, H, HS}. An $n$-connected domain $\Omega$ in the plane admits $2^{n-1}$ different spin structures. 
Since $\Omega$ admits a global frame, one of the spin structures, which we refer to as the trivial
one, can be identified with scalar single-valued functions. We shall show that the boundary value problem
(\ref{boundarycondition}) can always be solved for the trivial
spin structure. This will give a simple explicit characterization of $L^2$ boundary values of holomorphic functions.
In the process, we shall also find an analogue, for each of the spin structures, of the Hilbert transform
for multiply-connected domains. 

In this paper, we provide a functional-analytic proof of the existence and uniqueness of the solution to the problem \ref{boundarycondition}.

\smallskip
Our result can in principle be applied to the study of the Schramm's-SLE
curves SLE$\left(3\right)$ and SLE$\left(16/3\right)$ in finitely-connected geometries, for which
solutions of the boundary value problem (\ref{boundarycondition}) are crucial \cite{Izyurov}.

\section{Statement of the main results}

Let $\Omega$ be a bounded domain in $\mathbb{R}$ with smooth boundary $\p\Omega$. Let $\Ni$ and $\No$ be the unit inward and outward pointing normals
to the boundary $\p\Omega$. They can be identified with complex numbers
$\nui$ and $\nuo$ by writing,
\bea
\Ni=2\Re \mathfrak{e} (\nui{\p\over\p z}),
\qquad
\No=2\Re  \mathfrak{e} (\nuo{\p\over\p z}),
\eea
with $\nui=-\nuo$, $|\nui|=|\nuo|=1$. 

\medskip

Let $L^{2}\left(\partial\Omega\right)$
denote the space of $L^{2}$ complex-valued functions defined on $\partial\Omega$.
Even though this is clearly a (complex) Hilbert space, we will most
of the time view it as a \textbf{real} Hilbert space.
Define the (real-linear) projection operators 
$\mathsf{P}_{\mathrm{in}}:L^{2}\left(\partial\Omega\right)\to L^{2}\left(\partial\Omega\right)$
and $\mathsf{P}_{\mathrm{out}}:L^{2}\left(\partial\Omega\right)\to L^{2}\left(\partial\Omega\right)$ by
\begin{eqnarray*}
\mathsf{P}_{\mathrm{in}}\left[f\right]\left(z\right) & := & \frac{1}{2}\left(f\left(z\right)+\overline{\nu_{\mathrm{in}}}\left(z\right)\overline{f}\left(z\right)\right),\\
\mathsf{P}_{\mathrm{out}}\left[f\right]\left(z\right) & := & \frac{1}{2}\left(f\left(z\right)+\overline{\nu_{\mathrm{out}}}\left(z\right)\overline{f}\left(z\right)\right),
\end{eqnarray*}
Observe that the orthogonal projection
$\mathrm{Proj}_{e^{i\theta}\mathbb{R}}$
of $\mathbb{C}$
on the real line $e^{i\theta}\mathbb{R}$ in $\mathbb{C}$ can be expressed as
$\mathrm{Proj}_{e^{i\theta}\mathbb{R}}\left(\zeta\right)=\frac{1}{2}\left(\zeta+e^{2i\theta}\overline{\zeta}\right),\ \ \forall \zeta\in\mathbb{C},\forall\theta\in\mathbb{R}$.
Thus $\Pin[f](z)$ and $\Pout[f](z)$ are just the projections
of the complex number $f(z)$ on the two perpendicular lines in ${\bf C}$
defined by $\nui^{-{1\over 2}}(z)\mathbb{R}$ and $\nuo^{-{1\over 2}}(z)\mathbb{R}$.
As such, they are just twisted versions of the projections of complex numbers onto their real and imaginary parts. They provide
a simple way of formulating
boundary conditions of the form (\ref{boundarycondition}), e.g.,
\bea
\Im \mathfrak{m} (f(z)\nui^{{1\over 2}}(z))=0
\ \Leftrightarrow
\ \Pin[f](z)=0,
\qquad
z\in \p\Omega.
\eea
Clearly $\Pin^2=\Pin$, $\Pout^2=\Pout$,
and $\Pin+\Pout=\mathbf{Id}$. We can now define the real Hilbert subspaces
$L_{\rm in}^2(\p\Omega)$ and $L_{\rm out}^2(\p\Omega)$ by
\bea
&&
L_{\rm in}^2(\p\Omega)=Ker(\Pout)=Range(\Pin),
\\
&&
L_{\rm out}^2(\p\Omega)=Ker(\Pin)=Range(\Pout),
\eea
and we have the direct-sum decomposition,
\bea
L^2(\p\Omega)
=
L_{\rm in}^2(\p\Omega)\oplus L_{\rm out}^2(\p\Omega).
\eea

\smallskip
Let $H^2(\Omega)$ be the Hardy space of holomorphic functions on $\Omega$. It can be defined in several ways, and one way is as the Banach space of holomorphic functions $F(z)$ on $\Omega$ satisfying
\bea
{\rm sup}_{0<\delta<<1}
\int_{\p\Omega_\delta}|F(z)|^2 d\sigma(z)
<\infty,
\eea
where $\Omega_\delta$ is the subset of $\Omega$ consisting of points at a distance $>\delta$ from $\p\Omega$. For $\delta$ sufficiently small,
the orthogonal projection of $\p\Omega_\delta$ on $\p\Omega$ is a diffeomorphism, and functions on $\p\Omega_\delta$ can be identified with functions on $\p\Omega$. What is important for our purposes is the fact that for each function $F\in H^2(\Omega)$, the restrictions of $F$ to $\p\Omega_\delta$, viewed in this way as functions on $\p\Omega$, converge in $L^2(\p\Omega)$ and pointwise a.e. to
a function $R_{\p\Omega}F$. The ``restriction operator'' $R_\Omega$
is a bounded, injective, operator from $H^2(\Omega)$ to $L^2(\Omega)$. 

\medskip
Then our main result can be formulated as follows:

\begin{thm}
\label{10}
Let $\Omega$ be a bounded domain with smooth boundary. Then for any function
$f\in L^2(\p\Omega)$, there exists a unique function $F\in H^2(\Omega)$
satisfying the boundary condition
\bea
\Pin(R_\Omega(F)-f)=0.
\eea
\end{thm}

This theorem is essentially equivalent to another theorem, which 
is actually the one that we shall prove first. Let the operators $T_\Omega:H^2(\Omega)\to L_{\rm in}^2(\p\Omega)$
and $U_\Omega:H^2(\Omega)\to L_{\rm out}^2(\p\Omega)$ be defined by
\bea
T_\Omega:=\Pin\circ R_{\p\Omega},
\qquad
U_\Omega:=\Pout\circ R_{\p\Omega}.
\eea
It is useful to depict this graphically as
\[
\begin{aligned} &  & H^{2}\left(\Omega\right)\\
 & \swarrow_{T_{\Omega}} & \downarrow R_{\partial\Omega} & \searrow_{U_{\Omega}}\\
L_{\mathrm{in}}^{2}\left(\partial\Omega\right) & \leftarrow & L^{2}\left(\partial\Omega\right) & \rightarrow & L_{\mathrm{out}}^{2}\left(\partial\Omega\right)
\end{aligned}
\]

\smallskip

\begin{thm}
\label{2}
Let $\Omega$ be a bounded domain with smooth boundary. Then

\rm{(a)} the operators $T_\Omega:H^2(\Omega)\to L_{\rm in}^2(\p\Omega)$
and $U_\Omega:H^2(\Omega)\to L_{\rm out}^2(\p\Omega)$ are real-linear isomorphisms. 

\rm{(b)} Define the operator $W_\Omega: L_{\rm in}^2(\p\Omega)
\to L_{\rm out}^2(\p\Omega)$ by
\bea
W_\Omega=U_\Omega \circ T_\Omega^{-1}.
\eea
Then $W$ is a one-to-one and onto operator satisfying
\begin{eqnarray*}
W_{\Omega}\circ\left(\mathbf{j}\right)\circ W_{\Omega} & = & -\mathrm{Id}\\
\left(-\mathbf{j}\right)\circ W_{\Omega}\circ\left(\mathbf{j}\right) & = & T_{\Omega}\circ U_{\Omega}^{-1},
\end{eqnarray*}
where $\mathbf{j}$ denotes $i\cdot\mathrm{Id}$ (multiplication by
$i$ operator). 
\end{thm}

\medskip
To see how Theorem \ref{10} follows from Theorem \ref{2},
it suffices to observe that, if $T_\Omega^{-1}$ 
exists, then for any function $f\in L^2(\p\Omega)$,
the function $F=T_\Omega^{-1}(\Pin(f))$ satisfies the desired property.

\smallskip

Part (b) of Theorem {2} can be easily seen by tracing back the definitions
of the operators $T_\Omega$ and $U_\Omega$, once Part (a) has been
proved.
Thus we shall henceforth concentrate on the proof of Part (a) of Theorem
\ref{2}.

\medskip

The spaces $L_{\rm in}^2(\p\Omega)$ and $L_{\rm out}^2(\p\Omega)$
can be viewed as the analogues, for multiply-connected domains,
 of the spaces of $L^2$ functions with
only non-vanishing Fourier coefficients of respectively positive and negative
indices in the case of the unit disk.
The operator $W_{\Omega}$ is a variant of the Hilbert transform,
for the twisted line bundle $\nu_{\mathrm{out}}^{-\frac{1}{2}}(z)\mathbb{R}$
on the boundary of $\Omega$. 

\medskip
As we had stressed in the introduction, the analogue of the preceding theorem would fail if the projection operators $\Pin$ and $\Pout$ were
replaced by the projections on the trivial line bundles $\mathbb{R}\times\p\Omega$ and $i\mathbb{R}\times\p\Omega$ on the boundary of $\Omega$. In this case, 
there would have been an obstruction
of  a non-trivial finite-dimensional subspace. That this difficulty could be eliminated by considering boundary conditions of the form
(\ref{boundarycondition}) is a key insight provided by recent advances in the study of the critical Ising model. There is a discrete variant of the above theorems can be established explicitly for discrete fermions (see \cite{HK}, Section 13).

\section{Proof of Theorem \ref{2}}

We begin by proving the injectivity of the operators
$T_\Omega$ and $U_\Omega$.

\medskip

Let $F\in H^2(\Omega)$, and assume that $T_\Omega(F)=0$. If we set
$f=R_\Omega(F)$, this means that $f=\rho(z)\nui^{-{1\over 2}}(z)$
for some positive scalar function $\rho(z)$ on the boundary. We claim that
$\rho(z)=0$ identically, and hence $F=0$ identically in $\Omega$. Indeed, the holomorphicity of $F$ implies
\bea
\oint_{\p\Omega_\delta}F^2(z) dz=0
\eea
for all $0<\delta$ sufficiently small. But the convergence of the restrictions
of $F(z)$ to $\p\Omega_\delta$ to $f$, viewed as $L^2$ functions on $\p\Omega$
as explained in the previous section, implies in turn that
\bea
\oint_{\p\Omega}f^2(z) dz=0.
\eea
On the other hand, a key motivation for the boundary condition
(\ref{boundarycondition}) is the following identity,
\bea
\label{boundarycondition1}
\Re \mathfrak{e} (f^2(z)dz)=\rho^2(z)\Re \mathfrak{e} ({dz\over\nui(z)})
=
\rho^2(z) \,ds
\eea
where $ds$ is the element of arc-length along $\p\Omega$. 
This can be seen by picking a local defining function $r(z)$ for the boundary
$\p\Omega$. Then $\nu(z)={1\over |\nabla r|}(\p_xr+i\p_y r)$,
and
\bea
{dz\over \nu}= ds+i(\p_xr\,dy-\p_yr\,dx),
\eea
which implies (\ref{boundarycondition1}).
It follows that $\rho(z)=0$ identically, as was to be shown. The argument for the injectivity
of $U_{\Omega}$
is exactly the same. 

\medskip

The proof of Theorem {2} reduces then to the proof of the
surjectivity of the operators $T_\Omega$ and $U_\Omega$. This will be done
by induction on the number $n$ of components of the boundary
$\p\Omega$ of the domain $\Omega$. 
The precise statements that we shall prove are the following.
Let
\bea
\partial\Omega=\partial_{1}\Omega\cup\ldots\cup\partial_{n}\Omega, 
\eea
where the $\partial_j \Omega$'s are the connected components of $\partial \Omega$.

For each $j\in\left\{ 1,\ldots,n\right\} $, we denote by $H^{2}\left(\Omega,\partial_{j}\Omega\right)$
the subspace of $H^{2}\left(\Omega\right)$ defined by
\[
H^{2}\left(\Omega,\partial_{j}\Omega\right):=\left\{ f\in H^{2}\left(\Omega\right):T_{\Omega}\left(f\right)|_{\partial\Omega\setminus\partial_{j}\Omega}=0\right\} .
\]
Let $T_{\Omega}^{\partial_{j}\Omega}:H^{2}\left(\Omega\right)\to L_{\mathrm{in}}^{2}\left(\partial\Omega_{j}\right)$ be
the projection onto $L_{\mathrm{in}}^{2}\left(\partial\Omega_{j}\right)$
of $T_{\Omega}$,
\bea
f\mapsto T_\Omega^{\partial_{j}\Omega}(f)
=\left(T_{\Omega}\left(f\right)\right)|_{\partial\Omega_{j}}.
\eea
The restriction of $T_\Omega^{\p_j\Omega}$
to $H^{2}\left(\Omega,\partial_{j}\Omega\right)$ will be denoted by
$T_\Omega^j$. 

\medskip

Then Part (a) of Theorem \ref{2}
is an immediate consequence of the following two lemmas,
the first being the case $n=1$, and the second the induction step
from $n$ to $n+1$:

\begin{lem}
[Simply-connected case]
\label{lem:simply-connected-case}Let $\Omega$ be a simply-connected
domain. Then the mappings $T_{\Omega}:H^{2}\left(\Omega\right)\to L_{\mathrm{in}}^{2}\left(\partial\Omega\right)$
and $U_{\Omega}:H^{2}\left(\Omega\right)\to L_{\mathrm{out}}^{2}\left(\partial\Omega\right)$
are isomorphisms.
\end{lem}

\medskip

\begin{lem}[Induction step]
\label{lem:Induction-step}
Let $n\geq 1$,
and assume that for any $n$-connected smooth domain $\Omega$, the mapping
$T_{\Omega}:H^{2}\left(\Omega\right)\to L_{\mathrm{in}}^{2}\left(\partial\Omega\right)$
is an isomorphism.  Let $\Lambda$ be any $\left(n+1\right)$-connected
smooth domain. Then the operator
$T_\Lambda: H^2(\Lambda)\to L_{\rm in}^2(\p\Lambda)$
is an isomorphism. 
\end{lem}

\bigskip
We give now the proofs of Lemma \ref{lem:simply-connected-case} and Lemma
\ref{lem:Induction-step}. An essential ingredient is the following conformal invariance property. Its proof is straightforward,
since a conformal equivalence between two smooth domains extends to a diffeomorphism of the boundaries:

\medskip

\begin{lem}[Conformal equivalence]
\label{lem:conformal-equivalence}Let $\Lambda$ and $\Xi$ be two
conformally equivalent smooth domains and let $\psi:\Xi\to\Lambda$
be a conformal map. Let $\partial_{j}\Xi$ be a connected component
of $\partial\Xi$ and $\partial_{j}\Lambda:=\psi\left(\partial_{j}\Xi\right)$.
Then the following diagram commutes\textup{
\begin{eqnarray*}
H^{2}\left(\Lambda,\partial_{j}\Lambda\right) & \overset{T_{\Lambda}^{\partial_{j}\Lambda}}{\longrightarrow} & L_{\mathrm{in}}^{2}\left(\partial\Lambda_{j}\right)\\
\downarrow\Psi &  & \downarrow\Psi|_{L_{\mathrm{in}}^{2}\left(\partial\Xi{j}\right)}\\
H^{2}\left(\Xi,\partial_{j}\Xi\right) & \overset{T_{\Xi}^{\partial_{j}\Xi}}{\longrightarrow} & L_{\mathrm{in}}^{2}\left(\partial_{j}\Xi\right),
\end{eqnarray*}
}where the isomorphism $\Psi:H^{2}\left(\Lambda\right)\to H^{2}\left(\Xi\right)$
is defined by
\[
\left(z\mapsto f\left(z\right)\right)\mapsto\left(w\mapsto f\left(\psi_{j}\left(w\right)\right)\sqrt{\psi_{j}'\left(w\right)}\right).
\]

\end{lem}

\medskip
We observe that the square root $\sqrt{\psi_{j}'\left(w\right)}$ is well-defined
(up to a global harmless sign) even when $\Omega$ is multiply-connected
(see \cite{H}, Chapter 4). Thus the trivial spin structure is mapped into the trivial
spin structure under global conformal transformations.

\medskip
\noindent
{\it Proof of Lemma \ref{lem:simply-connected-case}:}  
If $\Omega$ is simply-connected, then there exists a conformal
equivalence between $\Omega$ and the unit disk $\mathbb{D}$ in
$\mathbb{C}$, which extends to a diffeomorphism
between the boundary $\p\Omega$ and the unit circle $\mathbb{S}$. 
By Lemma \ref{lem:conformal-equivalence},
it suffices to prove the desired statement when $\Omega=\mathbb{D}$
and $\p\Omega=\mathbb{S}$.

Let $\psi:\mathbb{D}\to\Omega$ be a conformal mapping. By Lemma \ref{lem:conformal-equivalence}
(defining the operator $\Psi$ as in that lemma), we have the following
commuting diagram
\begin{eqnarray*}
H^{2}\left(\Omega,\partial\Omega\right) & \overset{T_{\Omega}^{\partial\Omega}}{\longrightarrow} & L_{\mathrm{in}}^{2}\left(\partial\Omega\right)\\
\downarrow\Psi &  & \downarrow\Psi|_{L_{\mathrm{in}}^{2}\left(\partial\Omega\right)}\\
H^{2}\left(\mathbb{D},\mathbb{S}\right) & \overset{T_{\mathbb{D}}^{\mathbb{S}}}{\longrightarrow} & L_{\mathrm{in}}^{2}\left(\partial\mathbb{S}\right),
\end{eqnarray*}
We should now solve the problem on the unit disk: we should show that
$T_{\mathbb{D}}^{\mathbb{S}}=T_{\mathbb{D}}$ is invertible. 

So, let us construct the inverse $S_{\mathbb{D}}$. Let $f$ be a
function in $L_{\mathrm{in}}^{2}\left(\mathbb{S}\right)$. Let $\left(c_{k}\right)_{k\in\mathbb{Z}}$
be the Fourier coefficients of $f$, so that the Fourier series of
$f$ reads $\sum_{k\in\mathbb{Z}}c_{k}e_{k}$, where $e_{k}\left(\theta\right):=e^{ik\theta}$.
By definition of $L_{\mathrm{in}}^{2}\left(\mathbb{S}\right)$, we
have that $\mathsf{P}_{\mathrm{out}}\left[f\right]=\frac{1}{2}\left(f+e_{-1}f\right)=0$.
It is hence easy to see that we have $c_{k}+\overline{c_{-k-1}}=0$
for all $k\in\mathbb{Z}$. We define $S_{\mathbb{D}}\left(f\right):=g$,
where $g\in H^{2}\left(\mathbb{D}\right)$ is defined by
\[
g\left(z\right):=2\sum_{k=0}^{\infty}c_{k}\, z^{k}.
\]
This clearly defines a bounded operator. Let us check that $T_{\mathbb{D}}\left(g\right)=f$.
The Fourier series of $T_{\mathbb{D}}\left(g\right)$ reads
\[
\sum_{k=0}^{\infty}c_{k}\, e_{k}-\sum_{k=-\infty}^{-1}\overline{c_{-k-1}}\, e_{k}.
\]
The nonnegative Fourier coefficients $T_{\mathbb{D}}\left(g\right)$
are clearly the same as the ones of $f$, and using $c_{k}+\overline{c_{-k-1}}=0$
for all $k\in\mathbb{Z}$, we get that $T_{\mathbb{D}}\left(g\right)=f$.
Using exactly the same arguments, it is easy to check that $S_{\mathbb{D}}\circ T_{\mathbb{D}}$ is the identity.

\medskip

\noindent
{\it Remark.}

It is also possible to construct $S_{\mathbb{D}}$  by writing it explicitly as a convolution kernel.
This is in spirit closer to Ising model techniques: the convolution kernel corresponds then to a fermionic
correlator.

\medskip
We turn next to the proof of Lemma \ref{lem:Induction-step}. We need two simple
observations. The first is a superposition principle, which allows to
reduce the inversion of the operator $T_\Omega$ to the inversion of
operators $T_\Omega^{\p_j\Omega}$ associated to the components $ \partial_1 \Omega, \ldots, \partial_n \Omega$
of $ \partial \Omega $.

\medskip
\begin{lem}[Superposition]
\label{lem:superposition}Let $\Xi$ be an $\left(n+1\right)$-connected
domain with $\partial\Xi=\partial_{1}\Xi\cup\ldots\cup\partial_{n+1}\Xi$.
Suppose that for each $j\in\left\{ 1,\ldots,n+1\right\} $, the restriction
$T_{\Xi}^{j}:H^{2}\left(\Xi,\partial_{j}\Xi\right)\to L_{\mathrm{in}}^{2}\left(\partial_{j}\Xi\right)$
of the operator $T_{\Xi}^{\partial_{j}\Xi}$ (originally defined on
$H^{2}\left(\Xi\right)$) is an isomorphism. For each $j\in\left\{ 1,\ldots,n+1\right\} $,
denote by $S_{\Xi}^{\partial\Xi_{j}}:L_{\mathrm{in}}^{2}\left(\partial_{j}\Xi\right)\to H^{2}\left(\Xi\right)$
the inverse of $T_{\Xi}^{\partial_j}$, injected into $H^{2}\left(\Xi\right)$
(the range of $\left(T_{\Xi}^{\partial_j}\right)^{-1}$ is contained in $H^{2}\left(\Xi,\partial_{j}\Xi\right)$).
Then we have
\begin{eqnarray*}
T_{\Xi}^{\partial_{j}\Xi}\circ S_{\Xi}^{\partial_{j}\Xi} & = & \mathrm{Id}\,\,\,\,\forall j\\
T_{\Xi}^{\partial_{j}\Xi}\circ S_{\Xi}^{\partial_{k}\Xi} & = & 0\,\,\,\,\forall j\neq k,
\end{eqnarray*}
and $T_{\Xi}$ is invertible, with inverse $S_{\Xi}:=S_{\Xi}^{\partial_{1}\Xi}\oplus\ldots\oplus S_{\Xi}^{\partial_{j+1}\Xi}$
in the decomposition $L^{2}\left(\partial\Xi\right)=L^{2}\left(\partial_{1}\Xi\right)\oplus\ldots\oplus L^{2}\left(\partial_{n+1}\Xi\right)$. 
\end{lem}

The proof of this lemma is again straightforward.

\medskip

The second observation is 
the following version of the Riemann mapping theorem for multi-connected
domains:

\medskip

\begin{lem}[Riemann mapping theorem]
\label{lem:Riemann-mapping-theorem} 
Let $\Xi$ be an $n+1$-connected domain, with $n\geq 1$.
Then for any component $\p_j\Xi$ of the boundary $\p\Xi$,
there exists a conformal equivalence between $\Xi$ and 
$\Omega\setminus\overline{\mathbb{D}}$,
where $\Omega$ is an $n$-connected domain containing
the closure $\overline{\mathbb{D}}$ of the unit disk
$\mathbb{C}$, and $\p_j\Xi$ is mapped onto $\p\mathbb{D}$.
\end{lem}

\medskip
\noindent
{\it Proof of Lemma \ref{lem:Riemann-mapping-theorem}}:
Assume first that
$\p_j\Xi$ is an inner component of the boundary of $\Sigma$.
Let $D_1$ be the connected component of $\mathbb{C}\setminus\Sigma$
enclosed by $\p_j\Xi$. 
Let $I_{p}$ be the inversion map $z\to (z-p)^{-1}$, for any point $p\in\mathbb{C}$.
Choose a point $p_1\in D_1$ and a point $p_0\in
\mathbb{C}\setminus\overline\Omega$.
If we apply $I_{p_1}$, then the domain $\Sigma$ will be mapped to
a domain $I_{p_1}(\Sigma)$ lying within $\mathbb{C}\setminus
\overline{I_{p_1}(D_1)}$.
By the Riemann mapping theorem,
there exists a conformal equivalence $\Psi$
between the simply-connected domain $\mathbb{C}\setminus \overline{I_{p_1}(D_1)}$ and
the unit disk  $\mathbb{D}$, which extends to a diffeomorphism between
$I_{p_1}(\p_1\Sigma)$ and the unit circle $\mathbb{S}$. 
Then $\Sigma$ is conformally equivalent to the domain
$I_{I_{p_0}}\Psi I_{p_1}(\Sigma)$, one of whose boundary components
is the circle $I_{I_{p_0}}\Psi I_{p_1}(\p_1\Sigma)
=I_{I_{p_0}}(\mathbb{S})$. Translating and dilating so that this last circle
is the unit circle, we obtain the desired domain $\Omega$.
When $\p_j\Xi$ is the outer component of $\p\Xi$,
we can apply an inversion $I_p$ with respect to a point outside
$\overline{\xi}$ to transform $\xi$ into another domain with
$\p_j\Xi$ transformed into an inner component of the boundary.
This reduces the problem to the case already treated, and the proof
of Lemma \ref{lem:Riemann-mapping-theorem} is complete.

\medskip
The point of the two observations Lemma \ref{lem:superposition}
and Lemma \ref{lem:Riemann-mapping-theorem} is that, in conjunction with
the conformal equivalence property, it suffices to prove that each individual operator
$T_\Omega^{\p_j\Omega}$ for each fixed $j$, $1\leq j\leq n$,
is invertible, when $\p_j\Omega$ is an inner boundary and a unit circle.
This is the content of the next lemma, which is the hardest part of our argument,
and which will be proved in the next section:

\medskip
\begin{lem}[Key Lemma]
\label{lem:localization-in-disk-complement}
Let $n\geq2$. 
Let $\Omega$
be an $n$-connected domain containing the unit disk $\mathbb{D}$.
Let $\Xi$ be the $(n+1)$-connected domain defined
by $\Omega\setminus\overline{\mathbb{D}}$.
Assume that $T_{\Omega}:H^{2}\left(\Omega\right)\to L_{\mathrm{in}}^{2}\left(\partial\Omega\right)$
is an isomorphism.  Then $T_{\Xi}^{\mathbb{S}}:H^{2}\left(\Xi,\mathbb{S}\right)\to L_{\mathrm{in}}^{2}\left(\mathbb{S}\right)$
is an isomorphism. 
\end{lem}

\section{Proof of Lemma \ref{lem:localization-in-disk-complement}}

Recall that $\Xi=\Omega\setminus\overline{\mathbb{D}}$. Our goal is to construct
an inverse to the operator $T_\Xi^{\mathbb{S}}$, using the operator
$(T_\Omega)^{-1}$ and function theory on $\mathbb{S}$.

\subsection{Function theory on $\mathbb{S}$}

It is convenient to identify $L^2({\mathbb{S}})$ with $\ell^2(\mathbb{Z})$ by the Fourier transform, $\mathcal{F}:L^{2}\left(\mathbb{S}\right)\to\ell^{2}\left(\mathbb{Z}\right)$, defined by
\[
\left(f\in L^{2}\left(\mathbb{S}\right)\right)\mapsto\left(c_{k}\left(f\right):=\frac{1}{2\pi}\int_{0}^{2\pi}e^{-ikx}f\left(x\right)\mathrm{d}x\right)
\]
We denote by $\mathcal{F}^{-1}$ the inverse of $\mathcal{F}$.

Within $\ell^2(\mathbb{Z})$, let us introduce the following real-linear subspaces,
\begin{eqnarray*}
\ell_{-}^{2}\left(\mathbb{Z}\right) & := & \left\{ \left(c_{k}\right)_{k\in\mathbb{Z}}:c_{k}=0\,\,\forall k\geq0\right\} ,\\
\ell_{+}^{2}\left(\mathbb{Z}\right) & := & \left\{ \left(c_{k}\right)_{k\in\mathbb{Z}}:c_{k}=0\,\,\forall k<0\right\} ,\\
\ell_{\mathrm{in}}^{2}\left(\mathbb{Z}\right) & := & \left\{ \left(c_{k}\right)_{k\in\mathbb{Z}}:c_{k}-\overline{c_{-1-k}}=0\,\,\forall k\right\} ,\\
\ell_{\mathrm{out}}^{2}\left(\mathbb{Z}\right) & := & \left\{ \left(c_{k}\right)_{k\in\mathbb{Z}}:c_{k}+\overline{c_{-1-k}}=0\,\,\forall k\right\} .
\end{eqnarray*}
We denote by $\mathcal{P}_{\pm}$ the orthogonal projection on $\ell_{\pm}^{2}\left(\mathbb{Z}\right)$
and by $\mathcal{P}_{\mathrm{in}}:\ell^{2}\left(\mathbb{Z}\right)\to\ell_{\mathrm{in}}^{2}\left(\mathbb{Z}\right)$
the orthogonal projection on $\ell_{\mathrm{in}}^{2}$. In coordinates:
\begin{eqnarray*}
\mathcal{P}_{+}:\left(c_{k}\right)_{k\in\mathbb{Z}} & \mapsto & \left(\mathbf{1}_{\left\{ k\geq0\right\} }c_{k}\right)_{k\in\mathbb{Z}}\\
\mathcal{P}_{-}:\left(c_{k}\right)_{k\in\mathbb{Z}} & \mapsto & \left(\mathbf{1}_{\left\{ k<0\right\} }c_{k}\right)_{k\in\mathbb{Z}}\\
\mathcal{P}_{\mathrm{in}}:\left(c_{k}\right)_{k\in\mathbb{Z}} & \mapsto & \frac{1}{2}\left(c_{k}+\overline{c_{-k-1}}\right)\\
\mathcal{P}_{\mathrm{out}}:\left(c_{k}\right)_{k\in\mathbb{Z}} & \mapsto & \frac{1}{2}\left(c_{k}-\overline{c_{-k-1}}\right)
\end{eqnarray*}
Clearly, we have
\bea
\mathcal{P}_{-}\circ2\mathcal{P}_{\mathrm{in}}=\mathrm{Id}_{\ell_{-}^{2}\left(\mathbb{Z}\right)},
\qquad
2\mathcal{P}_{\mathrm{in}}\circ\mathcal{P}_{-}=\mathrm{Id}_{\ell_{\mathrm{in}}^{2}\left(\mathbb{Z}\right)}
\eea
as well as the commuting diagram,
\[
\begin{aligned}L^{2}\left(\mathbb{S}\right) & \underset{\mathsf{P}_{\mathrm{in}}}{\longrightarrow} & L_{\mathrm{in}}^{2}\left(\mathbb{S}\right)\\
\downarrow\mathcal{F} &  & \downarrow\mathcal{F}\\
\ell^{2}\left(\mathbb{Z}\right) & \underset{\mathcal{P}_{\mathrm{in}}}{\longrightarrow} & \ell_{\mathrm{in}}^{2}\left(\mathbb{Z}\right)
\end{aligned}
\]
where 
\begin{eqnarray*}
L_{\mathrm{in}}^{2}\left(\mathbb{S}\right): & = & \left\{ f\in L^{2}\left(\mathbb{S}\right):\Im\mathfrak{m}\left(f\left(e^{i\theta}\right)e^{i\theta/2}\right)=0\quad\mbox{for almost every }\theta\right\} .
\end{eqnarray*}
(this choice of notation is made as in our case the inner normal on
$\mathbb{S}$ is actually pointing \emph{towards the exterior of the
unit disk}\emph{\noun{ }}\emph{$\mathbb{D}$}, as $\mathbb{D}$ is
in the \emph{complement} of our domain $\Xi$). 

We also need the operator $\mathcal{J}:\ell^{2}\left(\mathbb{Z}\right)\to\ell^{2}\left(\mathbb{Z}\right)$ defined by
\bea
\label{J}
\mathcal{J}:\left(c_{k}\right)_{k\in\mathbb{Z}}\mapsto\left(\overline{c_{-1-k}}\right)_{k\in\mathbb{Z}}.
\eea
which exchanges $\ell_{+}^{2}\left(\mathbb{Z}\right)$ and $\ell_{-}^{2}\left(\mathbb{Z}\right)$, i.e.,
$\mathcal{J}\left(\ell_{\pm}^{2}\left(\mathbb{Z}\right)\right)=\ell_{\mp}^{2}\left(\mathbb{Z}\right)$. Finally, we set
\[
\mathcal{F}_{+}:=\mathcal{P}_{+}\circ\mathcal{F}
\qquad
\mathcal{F}_{-}:=\mathcal{P}_{-}\circ\mathcal{F}.
\]

\subsection{The operator $\Phi:\ell_-^2(\mathbb{Z})\to H^2(\mathbb{Z},\mathbb{S})$}

We come now to the main building block of the proof, which is the operator
$\Phi:\ell_-^2(\mathbb{Z})\to H^2(\mathbb{Z},\mathbb{S})$ constructed from
the operator $S_\Omega$ (which exists by induction hypothesis) 
and Fourier series on $\ell^2(\mathbb{Z})$ as follows.

\smallskip
For each negative integer $k\in \mathbb{Z}_-$, set
\bea
\varphi_{k}^{\Re \mathfrak{e}}\left(z\right):=z^{k}-S_{\Omega}\left(\mathsf{P}_{\mathrm{in}}\left[\zeta^{k}\right]\right),
\qquad
\varphi_{k}^{\Im}\left(z\right):=iz^{k}-S_{\Omega}\left(\mathsf{P}_{\mathrm{in}}\left[i\zeta^{k}\right]\right)\left(z\right),
\eea
where $S_{\Omega}:L_{\mathrm{in}}^{2}\left(\partial\Omega\right)\to H^{2}\left(\Omega\right)$ is
the inverse of $T_{\Omega}$ (which exists by assumption). Note that, while $\zeta^k$ has a pole at $0$, $\Pin[\zeta^k]$ is a well-defined $L^2$ function
on $\p\Omega$, and hence $S_\Omega(\Pin[\zeta^k])$ is well-defined as a holomorphic function on $\Omega$.
Thus the functions $\varphi_{k}^{\Re}$ and $\varphi_{k}^{\Im}$ are the
unique holomorphic functions on $\Omega\setminus\left\{ 0\right\} $
such that
\begin{eqnarray*}
\varphi_{k}^{\Re}\left(z\right)-z^{k}\mbox{ and }\varphi_{k}^{\Im}\left(z\right)-iz^{k} & \mbox{ are holomorphic in } & \Omega\\
\mathsf{P}_{\mathrm{in}}\left[\varphi_{k}^{\Re}\left(z\right)\right]=\mathsf{P}_{\mathrm{in}}\left[\varphi_{k}^{\Im}\left(z\right)\right]=0 & \mbox{on }\partial\Omega. & \mbox{ }
\end{eqnarray*}

\begin{lem} 
\label{lem:operator-Phi} 
Define the real-linear operator $\Phi:\ell^{2}_{-} ( \mathbb{Z} ) \to H^2(\Xi, \mathbb{S})$ by
\begin{equation}
\Phi\ :\ \left(c_{k}\right)_{k\in\mathbb{Z}}\mapsto\left(\sum_{k=-\infty}^{-1}\Re\mathfrak{e}\left(c_{k}\right)\varphi_{k}^{\Re}+\Im\mathfrak{m}\left(c_{k}\right)\varphi_{k}^{\Im}\right).\label{eq:series-phi}
\end{equation}

\rm{(a)} Then the operator $\Phi$ is well-defined and bounded
as a bounded operator from $\ell^2_{-}(\mathbb{Z})$ to
$H^{2}\left(\Xi,\mathbb{S}\right)$.

\rm{(b)} The operator $\Phi$ is invertible, and 
$\Phi^{-1}:H^2 (\Xi,\mathbb{S}) \to \ell^2_{-}(\mathbb{S})$ is equal to $\mathcal{F}R_{\p\Xi}^{\mathbb{S}}$.
\end{lem}

\begin{proof}[Proof of Lemma \ref{lem:operator-Phi}]
To prove Part (a), we have to show that the series defining $\Phi$
converges and is bounded in $H^2(\Xi,\mathbb{S})$ for $\left(c_{k}\right)_{k\in\mathbb{Z}}
\in \ell^2(\mathbb{Z})$. Since the boundary $\p\Omega$ of $\Omega$ lies
entirely within the region $\{|\zeta|>\rho\}$ for some fixed $\rho>1$,
the functions $\zeta^k$ decay exponentially fast for $k$ negative,
\bea
\|\zeta^k\|_{C^0(\p\Omega)}\leq \rho^{-k}.
\eea
By the assumption of Lemma \ref{lem:localization-in-disk-complement}, the operator $S_\Omega$ is bounded from $L^2(\p\Omega)$ to $H^2(\Omega)$. Thus we have
\begin{eqnarray*}
\|\left(\varphi_{k}^{\Re}-z^{k}\right)_{k<0}\|_{H^2(\Omega)} & \leq & C\rho^{-k}
\\
\|\left(\varphi_{k}^{\Im}-iz^{k}\right)_{k<0}\|_{H^2(\Omega)} & \leq & C\rho^{-k}
\end{eqnarray*}
for a constant $C$ independent of $k$. It follows that the series
\bea
\sum_{k=-\infty}^{-1}
\Re\mathfrak{e}\left(c_{k}\right)\left(\varphi_{k}^{\Re}-z^{k}\right)+\Im\mathfrak{m}\left(c_{k}\right)\left(\varphi_{k}^{\Im}-iz^{k}\right)
\eea
converges in $H^2(\Omega)$ and defines a function in $H^2(\Omega)\subset H^2(\Xi)$. Furthermore, by the Cauchy-Schwarz inequality, its $H^2(\Omega)$ norm
is bounded by
\bea
C\sum_{k=-\infty}^{-1}|c_k|\rho^{-k}
\leq
C(\sum_{k=-\infty}^{-1}|c_k|^2)^{1\over 2}(\sum_{k=-\infty}^{-1}\rho^{-2k})^{1\over 2}
\leq
C'(\sum_{k=-\infty}^{-1}|c_k|^2)^{1\over 2}.
\eea
Next, we consider the series
$\Re\mathfrak{e}\left(c_{k}\right)z^{k}+i\Im\mathfrak{m}\left(c_{k}\right)z^{k}$.
It converges uniformly on compact subsets and defines a holomorphic function
on $\Xi$. As we have seen, on the components of $\p\Xi$ which are also in $\p\Omega$, it converges exponentially fast. On the component $\mathbb{S}$
of $\p\Xi$, it is an $L^2$ function, since the functions $z^k$ form an orthonormal
system in $L^2(\mathbb{S})$,
\[
\left\Vert \sum_{k=-\infty}^{-1}\Re\mathfrak{e}\left(c_{k}\right)z^{k}+i\Im\mathfrak{m}\left(c_{k}\right)z^{k}\right\Vert _{L^2(\mathbb{S})}=\sum_{k=-\infty}^{-1}\left|c_{k}\right|^{2}<\infty.
\]
It follows that its $H^2(\Xi)$ norm is also bounded by $C\left( \sum_{k=-\infty}^{-1}|c_k|^2 \right)^{1\over 2}$. Altogether, the expression
\[
\Phi\left(\left(c_{k}\right)_{k\in\mathbb{Z}}\right)
=\sum_{k=-\infty}^{-1}\Re\mathfrak{e}\left(c_{k}\right)\left(\varphi_{k}^{\Re}-z^{k}\right)+\Im\mathfrak{m}\left(c_{k}\right)\left(\varphi_{k}^{\Im}-iz^{k}\right)+\Re\mathfrak{e}\left(c_{k}\right)z^{k}+i\Im\mathfrak{m}\left(c_{k}\right)z^{k},
\]
defines a function in $H^2(\Xi)$, whose $H^2(\Xi)$ norm
is bounded by$C\left( \sum_{k=-\infty}^{-1}|c_k|^2 \right)^{1\over 2}$. This shows that $\Phi$
is a well-defined and bounded operator, proving (a).

\smallskip

Next, we prove (b).
Checking that $\mathcal{F}_{-}\circ\Phi=\mathrm{Id}_{\ell_{-}^{2}\left(\mathbb{Z}\right)}$
is easy. To get that $\Phi\circ\mathcal{F}_{-}=\mathrm{Id}_{H^{2}\left(\Omega,\mathbb{S}\right)}$,
it just suffices to check that $\mathcal{F}_{-}$is injective. Suppose
that $f\in H^{2}\left(\Xi,\mathbb{S}\right)$ is such that $\mathcal{F}_{-}\left(f\right)$
is zero. Since it does not have any negative Fourier coefficients,
$f$ can be extended to a holomorphic function on $\Omega$. But $T_{\Omega}\left(f\right)=0$,
and hence by the injectivity of the operator $T_\Omega$
established earlier in Section 3, it follows that $f=0$.
\end{proof}

\subsection{The operators $\mathcal{O},\mathcal{Q}:\ell_-^2(\mathbb{Z})
\to \ell_-^2(\mathbb{Z})$}

Let $\mathcal{O}:\ell_-^2(\mathbb{Z})
\to \ell_-^2(\mathbb{Z})$ be the bounded operator defined by the composition
\[
\mathcal{O}:\ell_{-}^{2}\left(\mathbb{Z}\right)\underset{\Phi}{\longrightarrow}H^{2}\left(\Xi,\mathbb{S}\right)\underset{T_{\Xi}^{\mathbb{S}}}{\longrightarrow}L_{\mathrm{in}}^{2}\left(\mathbb{S}\right)\underset{\mathcal{F}_{-}}{\longrightarrow}\ell_{-}^{2}\left(\mathbb{Z}\right),
\]
Then we have the following formula:

\begin{lem}
\label{lem:O}
\rm{(a)} The operator $T_\Xi^{\mathbb{S}}$ is invertible with bounded inverse if and only if the operator $\mathcal{O}$ is invertible with bounded inverse.

\rm{(b)} The operator $\mathcal{O}$ can be expressed as
\bea
\mathcal{O}=\mathrm{Id}_{\ell_{-}^{2}\left(\mathbb{Z}\right)}+\mathcal{Q}
\eea
where $\mathcal{Q}$ is the operator on $\ell_-^2(\mathbb{Z})$ defined as the composition
\[
\mathcal{Q}:\ell_{-}^{2}\left(\mathbb{Z}\right)\underset{\Phi}{\longrightarrow}H^{2}\left(\Xi,\mathbb{S}\right)\underset{\mathcal{F}_{+}}{\longrightarrow}\ell_{+}^{2}\left(\mathbb{Z}\right)\underset{\mathcal{J}}{\longrightarrow}\ell_{-}^{2}\left(\mathbb{Z}\right),
\]
where\emph{ $\mathcal{J}$} is the operator of exchanging Fourier coefficients of positive and negative indices defined in Section 4.1.
\end{lem}

\begin{proof}[Proof of Lemma \ref{lem:O}] Part (a) follows immediately from
the fact that the other operators besides $T_\Xi^{\mathbb{S}}$ in the composition
defining the operator $\mathcal{O}$ are invertible with bounded inverses.

To prove Part (b),
let $\left(c_{k}\right)_{k\in\mathbb{Z}}\in\ell_{-}^{2}\left(\mathbb{Z}\right)$.
Set $f:=\Phi\left(\left(c_{k}\right)_{k\in\mathbb{Z}}\right)\in H^{2}\left(\Omega, \mathbb{S}\right)$, i.e.,
\[
f\left(z\right)=\sum_{k=-\infty}^{-1}\Re\mathfrak{e}\left(c_{k}\right)\varphi_{k}^{\Re}\left(z\right)+\Im\mathfrak{m}\left(c_{k}\right)\varphi_{k}^{\Im}\left(z\right).
\]
Set $g:=T_{\Xi}^{\mathbb{S}}\left(f\right)\in L_{\mathrm{in}}^{2}\left(\mathbb{S}\right)$. Thus
\begin{eqnarray*}
g\left(e^{i\theta}\right) & = & \frac{1}{2}\left(f\left(e^{i\theta}\right)+e^{-i\theta}\overline{f\left(e^{i\theta}\right)}\right)
\end{eqnarray*}
We compute the Fourier coefficients of $g$ with negative indices. Set $e_{k}\left(z\right):=z^{k}$
(defined on $\mathbb{S}$) and $\left\langle f,g\right\rangle :=\int_{0}^{2\pi}\overline{f\left(e^{i\theta}\right)}g\left(e^{i\theta}\right)\mathrm{d}\theta$.
We have
\begin{eqnarray*}
\left\langle e_{k},g\right\rangle  & = & \left\langle e_{k},f\right\rangle +\left\langle e_{k},e_{-1}\overline{f}\right\rangle \\
 & = & \left\langle e_{k},f\right\rangle +\left\langle e_{k+1},\overline{f}\right\rangle \\
 & = & \left\langle e_{k},f\right\rangle +\overline{\left\langle e_{-k-1},f\right\rangle }\\
 & = & \left\langle e_{k},\sum_{j=-\infty}^{-1}\Re\mathfrak{e}\left(c_{j}\right)\varphi_{j}^{\Re}+\Im\mathfrak{m}\left(c_{k}\right)\varphi_{j}^{\Im}\right\rangle +\overline{\left\langle e_{-k-1},f\right\rangle }\\
 & = & \sum_{j=-\infty}^{-1}\left(\Re\mathfrak{e}\left(c_{j}\right)\left\langle e_{k},\varphi_{j}^{\Re}\right\rangle +\Im\mathfrak{m}\left(c_{k}\right)\left\langle e_{k},\varphi_{k}^{\Im}\right\rangle \right)+\overline{\left\langle e_{-k-1},f\right\rangle },\\
 & = & c_{k}+\overline{\left\langle e_{-k-1},f\right\rangle }
\end{eqnarray*}
where in the before last equation, we use that $\varphi_{j}^{\Re}-z^{j}$
and $\varphi_{j}^{\Im}-iz^{j}$ are regular in the unit disk (and
hence have no negative Fourier coefficients). From there we get
$\mathcal{F}\circ T_{\Xi}^{\mathbb{S}}\circ\Phi\left(c_{k}\right)=c_{k}+\mathcal{Q}c_{k}$,
which is the desired result.
\end{proof}

We arrive now at the key property of the operator $\mathcal{Q}$, which is 
perhaps surprising in itself: 

\begin{lem}[Positive definiteness]
\label{lem:positive-definiteness}
The operator $\mathcal{Q}:\ell_-^2(\mathbb{Z})\to\ell_-^2(\mathbb{Z})$
is positive semi-definite 
with respect to the following real inner-product $\bullet_{\mathbb{R}}$ on $\ell_-^2(\mathbb{Z})$,
\[
\left(a_{k}\right)_{k}\bullet_{\mathbb{R}}\left(b_{k}\right)_{k}:=\sum_{k\in\mathbb{Z}}\Re\mathfrak{e}\left(a_{k}\right)\Re\mathfrak{e}\left(b_{k}\right)+\Im\mathfrak{m}\left(a_{k}\right)\Im\mathfrak{m}\left(b_{k}\right).
\]
\end{lem}

\begin{proof}[Proof of Lemma \ref{lem:positive-definiteness}]
Let $\left(c_{k}\right)_{k\in\mathbb{Z}}\in\ell_{-}^{2}\left(\mathbb{Z}\right)$
and set $f:=\Phi\left(\left(c_{k}\right)_{k}\right)\in H^{2}\left(\Xi,\mathbb{S}\right)$. Then, from the proof of Lemma \ref{lem:O}, we get:
\begin{eqnarray*}
\left(c_{k}\right)_{k}\bullet_{\mathbb{R}}\left(\mathcal{Q}\left(c_{k}\right)_{k}\right) & = & \sum_{k=-\infty}^{-1}\Re\mathfrak{e}\left(c_{k}\right)\Re\mathfrak{e}\left(\overline{\left\langle e_{-k-1},f\right\rangle }\right)+\Im\mathfrak{m}\left(c_{k}\right)\Im\mathfrak{m}\left(\overline{\left\langle e_{-k-1},f\right\rangle }\right)\\
 & = & \sum_{k=-\infty}^{-1}\Re\mathfrak{e}\left(\left\langle e_{k},f\right\rangle \right)\Re\mathfrak{e}\left(\left\langle e_{-k-1},f\right\rangle \right)-\Im\mathfrak{m}\left(\left\langle e_{k},f\right\rangle \right)\Im\mathfrak{m}\left(\left\langle e_{-k-1},f\right\rangle \right)
\end{eqnarray*}
On the other hand, by Fourier analysis, the conterclockwise-oriented
integral of $f^{2}$ on $\mathbb{S}$ gives 
\begin{eqnarray*}
\Re\mathfrak{e}\left(\frac{1}{2\pi i}\oint_{\mathbb{S}}f^{2}\left(z\right)\mathrm{d}z\right) & = & \left\langle e_{-1},\left(\sum_{k=-\infty}^{-1}\left\langle e_{k},f\right\rangle e_{k}\right)^{2}\right\rangle \\
 & = & \sum_{k=-\infty}^{-1}\Re\mathfrak{e}\left(\left\langle e_{k},f\right\rangle \left\langle e_{-k-1},f\right\rangle \right)\\
 & = & \sum_{k=-\infty}^{-1}\Re\mathfrak{e}\left(\left\langle e_{k},f\right\rangle \right)\Re\mathfrak{e}\left(\left\langle e_{-k-1},f\right\rangle \right)-\Im\mathfrak{m}\left(\left\langle e_{k},f\right\rangle \right)\Im\mathfrak{m}\left(\left\langle e_{-k-1},f\right\rangle \right)\\
 & = & \left(c_{k}\right)_{k}\bullet_{\mathbb{R}}\left(\mathcal{Q}\left(c_{k}\right)_{k}\right).
\end{eqnarray*}
Because $f\in H^2(\Xi)$, we can deform the integration contour 
of $f^2(z)dz$ as in Section 3 to get
\[
\oint_{\mathbb{S}}f^{2}\left(z\right)\mathrm{d}z=\oint_{\partial\Omega}f^{2}\left(z\right)\mathrm{d}z,
\]
where the orientation of the inner components of $\partial\Omega$
is clockwise, and the orientation of the outer component of $\partial\Omega$
is counterclockwise. But, as shown in Section 3, the fact that $f(z)$
satisfies the boundary contion on $\p\Omega$ implies that
$\Re({1\over i}f^{2}\left(z\right)\nu_{\mathrm{out}})\left(z\right)\geq 0$
on $\partial\Omega$. Hence
\[
\frac{1}{2\pi i}\oint_{\partial\Omega}f^{2}\left(z\right)\mathrm{d}z\geq0.
\]
We deduce that $\left(c_{k}\right)_{k}\bullet_{\mathbb{R}}\left(\mathcal{Q}\left(c_{k}\right)_{k}\right)\geq0$. This proves the lemma.
\end{proof}

\subsection{End of proof of Lemma \ref{lem:localization-in-disk-complement}}

Finally, we can complete the proof of Lemma \ref{lem:localization-in-disk-complement}: it suffices to
consider the operator 
\begin{eqnarray*}
\mathcal{O}\mathcal{O}^{T} & = & \left(\mathrm{Id}_{\ell_{-}^{2}\left(\mathbb{Z}\right)}+\mathcal{Q}\right)\left(\mathrm{Id}_{\ell_{-}^{2}\left(\mathbb{Z}\right)}+\mathcal{Q}^{T}\right)\\
 & = & \mathrm{Id}_{\ell_{-}^{2}\left(\mathbb{Z}\right)}+\left(\mathcal{Q}+\mathcal{Q}^{T}\right)+\mathcal{Q}\mathcal{Q}^{T}
\end{eqnarray*}
which is symmetric (for the scalar product $\bullet_{\mathbb{R}}$).
Its spectrum is bounded from below by $1$. Thus 
$\mathcal{O}\mathcal{O}^{T}$ is invertible, and we can write
\[
\mathcal{O}^{-1}=\mathcal{O}^{T}\left(\mathcal{O}\mathcal{O}^{T}\right)^{-1},
\]
which is clearly a bounded operator. As noted in Lemma \ref{lem:O}, Part (a),
the invertibility of $\mathcal{O}$ is equivalent to the invertibility of
the operator $T_{\Xi}^{\mathbb{S}}$. The proof of Lemma \ref{lem:localization-in-disk-complement},
and hence of Theorem \ref{2} is complete.

\section{Intrinsic Formulation of the Ising Boundary Condition}

The Ising boundary condition \ref{boundarycondition}
can be formulated intrinsically for spinors on an arbitrary Riemann surface
$\Omega$ with smooth boundary $\p\Omega$.
Recall that a spin structure $\delta$ on $\Omega$ is a holomorphic line bundle
$L_\delta$ on $\Omega$ with $L_\delta^2=K_\Omega$,
 where $K_\Omega$
is the canonical bundle of $\Omega$, that is, the bundle of $(1,0)$-forms over $\Omega$. Spinors on $\Omega$
with respect to the spin structure are then sections of $L_\delta$.
By abuse of notation, we shall often denote spinors by
$f(z)(dz)^{1\over 2}$. 

\smallskip
Fix a spin structure $L_\delta$. Choose any metric $ds^2=g_{\bar zz}dz d\bar z$
on $\Omega$ with $z,\bar z$ as isothermal coordinates, and let
$N(z)=2\Re (\nu(z){\p\over\p z})$ be the inward-pointing unit normal.
Then we say that a section
$f(z)(dz)^{1\over 2}$ of $L_\delta$ satisfies the Ising boundary condition if
\bea
\label{boundaryconditionglobal}
f(z)\nu(z)g_{\bar zz}^{1\over 2}=\overline{f(z)} \,(\nu(z)\bar\nu(z) g_{\bar zz})^{1\over 2}.
\eea
It is easily seen that this condition is equivalent to the condition
$\Im (f(z)\nu(z)^{1\over 2})=0$. However, it is clearly intrinsic: a square
root $(g_{\bar zz})^{1\over 2}$
of a metric $g_{\bar zz}$ on the surface $X$ is a well-defined metric
on the spin bundle $L_\delta^{-1}$. As such, it is a section
of the bundle $L_\delta\otimes \bar L_{\delta}$.
The left hand side is thus a section of $L_\delta\otimes L_\delta^{-2}
\otimes (L_\delta\otimes \bar L_{\delta})=\bar L_\delta$. Since $(\nu\bar\nu g_{\bar zz})$ is a scalar, the right hand side is also a section of $\bar L_\delta$, and the equation is intrinsic. Note that it is invariant under a Weyl scaling $g_{\bar zz}
\to e^{2\sigma(z)}g_{\bar zz}$ of the metric, so it is an equation that depends only
on the complex structure of $X$.

\section{Ellipticity of the Ising Boundary Condition}

The main goal of this section
is to show that, for generic even spin structures, the Ising boundary condition defines an elliptic
boundary value problem for the Cauchy-Riemann operator $\bar\p$.
This is essentially a consequence of classic arguments
for pseudo-differential operators, and we shall be brief.
For simplicity, we assume that $\Omega$
is an open subset of a compact Riemann surface $X$,
whose boundary $\p\Omega$ is a smooth, simple closed curve
in a domain holomorphically equivalent with the coordinate chart
$D=\{z\in\mathbb{C}; |z|<2\}$. We assume also that $X$ is
equipped with a spin structure $\delta$, and we equip $\Omega$ with the induced spin structure. For generic even spin structures, the dimension of the space of holomorphic spinors on $X$ is $0$.

\smallskip
We adapt the method of multiple-layer potentials. Let $S_\delta(z,w)(dz)^{1\over 2}
\otimes (dw)^{1\over 2}\in L_\delta(z)
\otimes L_\delta(w)$
be the Szeg\"o kernel for $L_\delta$,
where $L_\delta$ is the spin bundle defined by $\delta$,
$L_\delta^2=K_X$, and $K_X$ is the line bundle
of $(1,0)$-forms on $X$. Then $S_\delta(z,w)$ has exactly one simple pole in $z$ at 
$w$ when $\delta$ is a generic even spin structure (see e.g. \cite{DP, F}). 
Using the coordinate system on the chart $D$, we can write any section
of $L_\delta$ over $\p\Omega$ as
$f(w)(dw)^{1\over 2}$.
Thus we can define the operator
\bea
\label{F}
F(z)=\left({1\over\pi i}\int_{\p\Omega}S_\delta(w,z) f(w) dw\,(dz)^{1\over 2}\right)
\eea
which maps sections of $L_\delta$ over $\p\Omega$ to sections of $L_\delta$
over $\Omega$. Clearly $F(z)$ is holomorphic on $\Omega$. 

\medskip
(a) The boundary values of the spinor $F(z)$ exist and are given by
\beastar
{\rm lim}_{z\to w_0}
F(z)=\left ({\rm lim}_{\e\to 0}
{1\over\pi i}\int_{\p\Omega_\e(w_0)}S_\delta(w,w_0)f(w)dw \, (dw_0)^{1\over 2}-f(w_0)(dw_0)^{1\over 2} \right)
\eeastar
for any $w_0\in\p\Omega$. Here $\p\Omega_\e(w_0)$ is the complement in $\p\Omega$ of the disk centered at $w_0$ and of radius $\e$ with respect to the Euclidian metric. This follows by 
the well-known Plemelj arguments (see e.g.
\cite{M}): the function $f(w)$ in (\ref{F}) can be replaced by $f(w)-f(w_0)$, since
the integral
of $S_\delta(w,z) dw$ over the contour $\p\Omega$ 
for fixed $z\in\Omega$ can be deformed to the origin of the disk $D$. 
The resulting integral over $\p\Omega$ converges when $z\to w_0$ and can be replaced by the limit of integrals over $\p\Omega_\e(w_0)$. The contribution of the term $f(w_0)$ can now be evaluated separately: the contour $\p\Omega_\e(w_0)$ can be viewed as the difference between
a closed contour $\p\tilde\Omega_\e(w_0)$
consisting of $\p\Omega_\e(w_0)$ completed by a small half-loop $C_\e(w_0)$ of radius $\e$ around $w_0$, in the exterior of $\Omega$, and the half-loop $C_\e(w_0)$ itself. The contribution over $\p\tilde\Omega_\e(w_0)$ is again $0$ by the holomorphicity of $S_\delta(w,w_0)$, while the one over $C_\e(w_0)$
can be calculated exactly in the limit $\e\to 0$, using the short-distance asymptotics of $S_\delta(z,w)=(z-w)^{-1}+O((z-w)^3)$.

\smallskip

(b) Let $H$ be the operator on $\p\Omega$ defined by the first expression
on the right hand side of (\ref{F}). Then $H$ is a pseudo-differential operator of order $0$. Its leading term maps real functions to real functions, and its symbol $\sigma(z,\xi)$, if we parametrize $\p\Omega$ by the arc-length with respect to the Euclidian metric, is given by
\bea
\label{sigma}
\sigma(z,\xi)=i({\rm sgn}\,\xi).
\eea
This statement is local. Since we can work near the diagonal $z=w$
and drop smoothing errors, we can replace $S_\delta(z,w)$ by $(z-w)^{-1}$. 
If we parametrize $\p\Omega$ by the arc-length from a fixed point $P\in\p\Omega$,
$s\to w(s)\in\p\Omega$,
we can express $H$ as a one-dimensional singular integral operator with kernel
\bea
K(s,s_0)={t(s)\over w(s)-w(s_0)}={t(s)\over (s-s_0)(t(s_0)+(s-s_0)E(s,s_0))}
\eea
where $t(s)=dw/ds$, and $E(s,s_0)$ is a smooth function. The last expression can be recognized as $(s-s_0)^{-1}$ up to a smooth kernel. Thus $H$ is, up to smoothing errors, just the classic Hilbert transform, and it is well-known that
(\ref{sigma}) is its symbol.

\smallskip
(c) The Ising boundary condition (\ref{boundaryconditionglobal}) can be interpreted
as the problem of finding $F$ with boundary values admitting
a given projection along each direction $\nu^{1\over 2}$.
To check ellipticity of this boundary value problem, we can again work locally
and restrict ourselves to the terms of leading order.
Then the complex number $\nu(z)g_{\bar zz}^{1\over 2}$
has modulus $1$, and can be expressed locally as 
$\nu(z)g_{\bar zz}^{1\over 2}=e^{2i\theta}$ for some real-function $\theta$ on $\p\Omega$. The (signed) length of the projection of a complex number $\zeta$
on the line $e^{i\theta}\mathbb{R}$ is given by
\bea
{1\over 2}e^{-i\theta}(\zeta+e^{2i\theta}\bar\zeta)
={1\over 2}(e^{-i\theta}\zeta+e^{i\theta}\bar\zeta).
\eea
Thus the ellipticity of the Ising boundary condition is just the ellipticity of the operator on real functions
\bea
f\to Mf=:{1\over 2}(e^{-i\theta}(iH-I)+e^{i\theta}(-iH-I))f
=({\rm sin}\,\theta\,H-{\rm cos}\,\theta)f.
\eea
Since the principal
symbol of the operator $M$ is $i({\rm sgn}\,\xi)\,{\rm sin}\,\theta-{\rm cos}\,\theta$, which has norm $1$, the ellipticity of $M$ follows at once.

\smallskip
(d) As a consequence, the operator $M$ admits a parametrix which
is a pseudodifferential operator of order $0$. In particular, it is bounded on
Schauder spaces and on Sobolev spaces. Thus, when the operator $T_\Omega^{-1}$ exists, it is bounded on Schauder and on Sobolev spaces.

(e) In general, the dimensions of the kernel of $M$
and of its co-range are finite-dimensional.

\section{Canonical Metrics}

The solvability of the Ising boundary condition yields a new canonical metric
for smooth multi-connected domains in ${\bf C}$.	
In the case of the trivial spin structure, this metric corresponds to the energy density one-point function of the model, with locally constant +/- boundary conditions (see \cite{H}, Chapter 7).

\smallskip
Let $\Omega\subset\mathbb{C}$ be a multi-connected domain with smooth boundary $\p\Omega$. Theorem \label{1} implies the existence and uniqueness
of the solution to the boundary value problem
\bea
\p_{\bar z}G(z,w)=\delta(z,w)\ \ {\rm in}\ \Omega,
\quad
\Im(G(\cdot,w)\sqrt\nu)=0
\eea
for any given $w\in \Omega$. We set
\bea
\ell(w)={\rm lim}_{z\to w}(G(z,w)-{1\over z-w}).
\eea
It is then easy to show, by a similar argument as in the proof of the uniqueness part
of Theorem \ref{1}, that $\ell(w)$ is always a strictly positive number. Thus
\bea
ds_\Omega^2=: \ell(w)^2 dw d\bar w
\eea
defines a metric on $\Omega$. Lemma \ref{lem:conformal-equivalence} implies that, for any conformal equivalence $\Phi$, 
\bea
ds_{\Phi(\Omega)}^2=\Phi_*(ds_\Omega^2).
\eea
In this sense, the ``Ising energy metric" $ds^2$ is a canonical metric, which is actually different from the many other canonical metrics known in the literature. This can be verified explicitly in the case of an annulus $\Omega=\{z\in\mathbb{C};
1<|z|<R\}$ for some fixed $R>1$. Then it is not difficult to verify that the Ising model
metric is given by
\bea
\ell(w)=\sum_{n\in\mathbb{Z}}{1\over 1+R^{2n+1}}|w|^{2n}
\eea
for the even spin structure (the spinors are anti-periodic when one goes around
the circles centered at $0$), and
\bea
\ell(w)=\sum_{n\in\mathbb{Z}}{1\over 1+R^{2n}}|w|^{2n+1}
\eea
for the odd spin structure (the spinors are now periodic). On the other hand, the Bergman metric $K(z)$ is given by
\bea
K(z)={1\over\pi \log R^2}|z|^{-2}
+{1\over\pi}\sum_{n\in\mathbb{Z}}{1\over( R^p|z|^2-R^{-p})^2}
\eea
while the Robin metric is given in terms of the $\theta$-function
\bea
\theta({1\over 2\pi i}\log{R\over |w|^2}+{1\over 2}|{i\over\pi}\log R),
\eea
up to a factor independent of $w$ (in fact, a Dedekind function in $R$).

\bigskip
\subsection*{Acknowledgements}
C.H. would like to thank Dmitry Chelkak and Stanislav Smirnov for introducing him to the complex analysis questions related to the Ising model and sharing many ideas about this subject, 
as well as St\'{e}phane Benoist, Julien Dub\'{e}dat, Konstantin Izyurov and Igor Krichever for interesting conversations during the preparation of this manuscript.

C.H. is supported by NSF grant DMS-1106588 and the Minerva Foundation.

D.H.P. is supported by NSF grant DMS-07-57372.

\end{document}